\documentclass[final,1p,times,sort&compress]{elsarticle}%
\usepackage{amsmath}
\usepackage{amssymb}
\usepackage{amscd}
\usepackage{amsfonts}
\usepackage{enumerate}
\usepackage{mathrsfs}
\usepackage{xy}
\usepackage{stmaryrd}
\usepackage{graphicx}
\usepackage{fancybox}
\usepackage{portland}
\usepackage{color}
\usepackage{multirow}
\usepackage{exscale}
\usepackage{enumitem,array}%
\usepackage{subfigure}
\usepackage{extarrows}
\usepackage{amsfonts}
\usepackage{slashbox}

\newtheorem{mytheo}{Theorem}[section]

\newtheorem{algo}[mytheo]{Algorithm}

\newtheorem{rem}[mytheo]{Remark}

\newcommand{\abs}[1]{\left\vert#1\right\vert}

\newcommand{\eps}{\varepsilon}

\newcounter{remark}

\newcounter{problem}

\makeatletter
\def\@upcite#1#2{\textsuperscript{[{#1\if@tempswa , #2\fi}]}}
\makeatother
\newenvironment{proof}{\vspace{1ex}
	{\it Proof. }\hspace{0.3em}}{\vspace{1ex}} \journal{ }
\begin{document}
	\begin{frontmatter}
		\title{A novel class of explicit energy-preserving splitting methods for charged-particle dynamics}

		\author[1]{Xicui Li}
		\author[1]{Bin Wang\corref{cor1}}

		\address[1]{ School of Mathematics and Statistics, Xi'an Jiaotong University, 710049 Xi'an, China}

		\ead{lixicui@stu.xjtu.edu.cn,wangbinmaths@xjtu.edu.cn}\cortext[cor1]{Corresponding author.}

		
		\begin{abstract}
			In this letter, {based on the exponential scalar auxiliary variable technology}, we propose and study a new class of explicit energy-preserving splitting methods for solving the charged-particle dynamics. The energy-preserving property of these methods {is rigorously} analysed. {We also provide the error estimates for the new methods. Numerical computations are presented, which confirm the effectiveness  and superiority of these novel methods  in comparison with the standard scalar auxiliary variable approach.}
		\end{abstract}
		
	\end{frontmatter}
	
	\section{Introduction}
	
	In this letter, we focus on the charged-particle dynamics (CPD) {\cite{add1,2020Long}}
	\begin{equation}\label{CPD equ}
		\begin{split}
			&\ddot{x}(t)=\dot{x}(t)\times B(x(t))+E(x(t)),\ \ x(0)=x^0, \ \dot{x}(0)=\dot{x}^0,\ t\in [0,T],
		\end{split}
	\end{equation}
	with the position  $x(t)\in\mathbb{R}^3$ and the velocity $v(t):=\dot{x}(t)\in\mathbb{R}^3$ {of a particle}. The particle, whose initial values are $x^0$ and $v^0:=\dot{x}^0\in \mathbb{R}^3$, is moving in an non-uniform magnetic field $B(x(t))=(b_1(x),b_2(x),b_3(x))^\intercal$ and an electric field $E(x)=-\nabla_x U(x)$ with some scalar potential $U(x)$. The total energy of the CPD is conserved along the solution of \eqref{CPD equ} and it is {in the form of}
	\begin{equation}\label{H(x,v)}
		H\left(x(t),v(t)\right)=\frac{1}{2}\abs{v(t)}^2+U(x(t)),\quad t\geq0,
	\end{equation}
	with the Euclidean norm $\abs{\cdot}$.
	
	In recent years, the energy-preserving (EP) property of numerical methods for solving CPD has gained {considerable} attention, and numerous EP methods (\cite{L. Brugnano2020,L. Brugnano2019,Li-AML,Chacon,2020Error,22Li,Wang2021}) have been constructed and analysed to solve this system. However, all these methods are implicit {and a nonlinear iteration is needed in practical computations. Thus  it is time-consuming to adopt them to calculate the CPD  in comparison with explicit methods.} In order to improve the computational efficiency of EP methods, the {sacalar auxiliary variable (SAV)} \cite{2018Shen,2019Shen,19Li} approach has been considered to formulate a class of linearly implicit splitting EP schemes {(see e.g. in \cite{23Li})} which are shown to be  more efficient. By introducing an auxiliary scalar, the SAV approach is proposed for constructing energy stable schemes for a broad class of gradient flows \cite{79Allen,58Cahn} and has been effectively used in a number of conservative systems, such as Hamiltonian systems \cite{19Cai,20Cai}. The standard SAV approach for the system \eqref{CPD equ} is formulated under the condition that the scalar potential $U(x)$ is bounded from below, i.e., there exists a positive constant $c_0$ such that $U(x)\ge -c_0$. {Then} we can introduce a scalar $s(t)=\sqrt{U(x)+C_0}$ ($C_0>c_0$) {and {apply} splitting {technology} to obtain {numerical schemes (see \cite{23Li}). For example, { using SAV approach}, the first-order scheme (denoted by S1-SAV) has been derived in \cite{23Li}, which reads}
		\begin{equation}\label{S1SAV}
			\begin{aligned}	&x^{n+1}=x^n+he^{h\tilde{B}(x^n)}v^n+\frac{h^2}{2}\frac{E(\hat{x}^{n+\frac{1}{2}})}{\sqrt{U(\hat{x}^{n+\frac{1}{2}})+C_0}}s^{n+\frac{1}{2}},\ \ v^{n+1}=e^{h\tilde{B}(x^n)}v^n+h\frac{E(\hat{x}^{n+\frac{1}{2}})}{\sqrt{U(\hat{x}^{n+\frac{1}{2}})+C_0}}s^{n+\frac{1}{2}},\\
				&s^{n+1}
				=s^n-\frac{(x^{n+1}-x^n)^\intercal E(\hat{x}^{n+\frac{1}{2}})}{2\sqrt{U(\hat{x}^{n+\frac{1}{2}})+C_0}},
			\end{aligned}
		\end{equation}
		with the approximate term $\hat{x}^{n+\frac{1}{2}}=x^n+\frac{h}{2}e^{h\tilde{B}(x^n)}v^n$ and the notation $s^{n+\frac{1}{2}}=\frac{s^{n+1}+s^n}{2}$, {where $h$ is the time stepsize}. {According to the analysis in  \cite{23Li}, we know }that the S1-SAV {exactly }preserves the modified energy $\tilde{H}(v,s)=\frac{1}{2}\abs{v}^2+s^2-C_0$ at the discrete level.  { In addition, for the schemes presented in \cite{23Li},} there are three aspects that can be improved.
		\begin{itemize}
			\item The lower bound condition {is required for the  scalar potential  $U(x)$, which
				is not always satisfied, such as the case that}  $U(x)=x_1^3-x_2^3+x_1^4/5+x_2^4+x_3^4$.
			\item In the implementation of {the methods of \cite{23Li}, since the scheme is linearly implicit,} the calculation of solution variables and the
			auxiliary variable can not be decoupled. Thus we have to determine the inner product $\left(E(\hat{x}^{n+\frac{1}{2}})\right) ^\intercal x^{n+1}$ before computing $x^{n+1}$, which would become more {complicated} for high-order SAV schemes.
			\item {It is obvious that the scalar $s(t)>0$. However it is  difficult} to guarantee that {its numerical solution presented in \eqref{S1SAV} also satisfies} $s^n>0$.
		\end{itemize}
		
		Motivated by these points,
		we propose and study a new kind of explicit energy-preserving splitting methods for charged-particle dynamics. The new methods
		do not need the lower bound condition of  $U(x)$ and are {completely explicit} which {makes} the methods can be implemented more efficiently. Moreover, the new methods can share the property $s(t)>0$ of the scalar $s(t)$.
		The proposed methods are formulated based on   the exponential scalar auxiliary variable (E-SAV) technology which was firstly presented in \cite{20Liu} and has been  popular in the formulation of effective methods for various  phase field models such as Hamiltonian PDEs \cite{22Bo} and Allen-Cahn type equations \cite{22Ju}. For the system of CPD, using its specific structure and the E-SAV technology, we can} get rid of the assumption of the nonlinear potential scalar in the SAV approach, and this yields {totally} explicit energy stable schemes. As a result, we obtain a decoupled scheme, and the time consumption of E-SAV is more efficient than SAV.

	The rest of this letter is organized as follows. By introducing an exponential scalar auxiliary, Section \ref{method} presents two explicit splitting methods and analyzes their energy-preserving properties and global error bounds. In Section \ref{ne}, a numerical experiment is given to demonstrate the energy, cputime and accuracy behaviour of the obtained methods in comparison with the method S1-SAV. {Section \ref{conc} includes the conclusion of} this letter.
	
	\section{Numerical methods and their properties}\label{method}
	In this section, we {first consider an exponential scalar auxiliary variable: $r(t)=\exp(U(x))$. Then the equation \eqref{CPD equ} can be transformed into}
	{$$\frac{d}{dt}\begin{pmatrix}
			x\\v\\r
		\end{pmatrix}=\begin{pmatrix}
			v\\\tilde{B}(x)v+\frac{E(x)}{\exp(U(x))}r\\-\exp(U(x))\dot{x}^\intercal E(x)
		\end{pmatrix}  \ \ \textmd{with}\ \ \tilde{B}(x)=\begin{pmatrix}
			0&b_3(x)&-b_2(x)\\
			-b_3(x)&0&b_1(x)\\
			b_2(x)&-b_1(x)&0
		\end{pmatrix}.$$} Now we reformulate  the above equation as
	\begin{equation}\label{new CPD}
		\frac{d}{dt}\begin{pmatrix}
			x\\v\\\ln(r)
		\end{pmatrix}=\begin{pmatrix}
			v\\\tilde{B}(x)v+\frac{E(x)}{\exp(U(x))}r\\-\frac{\dot{x}^\intercal E(x)}{\exp(U(x))}r
		\end{pmatrix},\quad \begin{pmatrix}
			x(0)\\v(0)\\\ln(r(0))
		\end{pmatrix}={\begin{pmatrix}
				x^0\\v^0\\\ln(r^0)
		\end{pmatrix}},
	\end{equation}
	where $v^0:=\dot{x}^0$ and $r^0:=\exp(U(x^0))$. In order to {obtain} the numerical solution of the system \eqref{new CPD}, we split it into two subflows:
	\begin{equation}\label{subflows}
		\frac{d}{dt}\begin{pmatrix}
			x\\v\\\ln(r)
		\end{pmatrix}=\begin{pmatrix}
			0\\\tilde{B}(x)v\\0
		\end{pmatrix},\quad \frac{d}{dt}\begin{pmatrix}
			x\\v\\\ln(r)
		\end{pmatrix}=\begin{pmatrix}
			v\\\frac{E(x)}{\exp(U(x))}r\\-\frac{\dot{x}^\intercal E(x)}{\exp(U(x))}r
		\end{pmatrix}.
	\end{equation}
	For the first subflow, which is linear, it is easy to get its exact solution $\Phi_t^L$: $\begin{pmatrix}
		x(t)\\v(t)\\\ln(r(t))
	\end{pmatrix}=\begin{pmatrix}
		x(0)\\e^{\tilde{B}(x(0))}v(0)\\\ln(r(0))
	\end{pmatrix}.$
	Subsequently,  {for  the second subflow of \eqref{subflows}, we {denote $F(x,r):=\frac{E(x)}{\exp(U(x))}r$ and }{consider the  following explicit numerical} propagator $\Phi_h^{NL}$:
		\begin{equation}\label{ESAV}
			x^{n+1}=x^n+hv^n+\frac{h^2}{2}{F(\tilde{x}^{n+\frac{1}{2}},\tilde{r}^{n+\frac{1}{2}})},\ \ v^{n+1}=v^n+h {F(\tilde{x}^{n+\frac{1}{2}},\tilde{r}^{n+\frac{1}{2}})},\ \ln (r^{n+1})=\ln (r^n)- (x^{n+1}-x^n)^\intercal {F(\tilde{x}^{n+\frac{1}{2}},\tilde{r}^{n+\frac{1}{2}})},
		\end{equation}
		where $h$ is the time stepsize,
		$\tilde{x}^{n+\frac{1}{2}}=x^n+\frac{h}{2}v^n$ and $\tilde{r}^{n+\frac{1}{2}}=r^n-\frac{h}{2}\exp(U(\tilde{x}^{n+\frac{1}{2}}))(v^n)^\intercal E(\tilde{x}^{n+\frac{1}{2}})$ are {respectively numerical   approximations for $x(t_{n+\frac{1}{2}})$ and $r(t_{n+\frac{1}{2}})$ with the accuracy $\mathcal{O}(h^2)$ and $t_{n+\frac{1}{2}}=\big(n +\frac{1}{2}\big)h$}. Finally, we can obtain $r^{n+1}=\exp\left[ \ln (r^n)- (x^{n+1}-x^n)^\intercal {F(\tilde{x}^{n+\frac{1}{2}},\tilde{r}^{n+\frac{1}{2}})}\right]. $
		
		{With the above preparations, we are now in the position to present the scheme of the explicit energy-preserving splitting methods.}
		\begin{algo}\label{algo}{(Explicit Energy-Preserving Splitting Methods)}
			For the sake of brevity, we denote the numerical solution as $x^n\approx x(t_n)$, $v^n\approx v(t_n)$. On the basis of the composition of $\Phi_h^L$ and $\Phi_h^{NL}$, we derive the following explicit schemes, such as the first order splitting scheme $\Phi_h^1=\Phi_h^{NL}\circ\Phi_h^L$ (S1-ESAV):
			\begin{equation}\label{S1-ESAV}
				\begin{aligned}	&x^{n+1}=x^n+he^{h\tilde{B}(x^n)}v^n+\frac{h^2}{2}
					{F(\hat{x}^{n+\frac{1}{2}},\hat{r}^{n+\frac{1}{2}})},\ \ v^{n+1}=e^{h\tilde{B}(x^n)}v^n+h{F(\hat{x}^{n+\frac{1}{2}},\hat{r}^{n+\frac{1}{2}})},\\
					&\ln (r^{n+1})=\ln (r^n)- (x^{n+1}-x^n)^\intercal  {F(\hat{x}^{n+\frac{1}{2}},\hat{r}^{n+\frac{1}{2}})},
				\end{aligned}
			\end{equation}
			with the approximate terms $\hat{x}^{n+\frac{1}{2}}=x^n+\frac{h}{2}e^{h\tilde{B}(x^n)}v^n$, $\hat{r}^{n+\frac{1}{2}}=r^n-\frac{h}{2}\exp(U(\hat{x}^{n+\frac{1}{2}}))\left( e^{h\tilde{B}(x^n)}v^n\right) ^\intercal E(\hat{x}^{n+\frac{1}{2}})$ and the second order Strang splitting scheme $\Phi_h^2=\Phi_{h/2}^L\circ\Phi_h^{NL}\circ\Phi_{h/2}^L$ (S2-ESAV):
			\begin{equation}\label{S2-ESAV}
				\begin{aligned} &x^{n+1}=x^n+he^{\frac{h}{2}\tilde{B}(x^n)}v^n+\frac{h^2}{2}{F(\widehat{x}^{n+\frac{1}{2}},\widehat{r}^{n+\frac{1}{2}})},\ \ v^{n+1}=e^{\frac{h}{2}\tilde{B}(x^{n+1})}\left[ e^{\frac{h}{2}\tilde{B}(x^n)}v^n+h{F(\widehat{x}^{n+\frac{1}{2}},\widehat{r}^{n+\frac{1}{2}})}\right] ,\\
					&\ln (r^{n+1})=\ln (r^n)- (x^{n+1}-x^n)^\intercal {F(\widehat{x}^{n+\frac{1}{2}},\widehat{r}^{n+\frac{1}{2}})},
				\end{aligned}
			\end{equation}
			with the approximate terms $\widehat{x}^{n+\frac{1}{2}}=x^n+\frac{h}{2}e^{\frac{h}{2}\tilde{B}(x^n)}v^n$ and  $\widehat{r}^{n+\frac{1}{2}}=r^n-\frac{h}{2}\exp(U(\widehat{x}^{n+\frac{1}{2}}))\left( e^{\frac{h}{2}\tilde{B}(x^n)}v^n\right) ^\intercal E(\widehat{x}^{n+\frac{1}{2}})$.
			{These two methods are denoted by SESAVs.}
		\end{algo}
		
		{It is noted that higher-order schemes can be produced by  applying the Triple Jump splitting to S2-ESAV, but we skip this in the letter for brevity. In what follows, we study the energy-preserving property} of these two splitting schemes.
		
		\begin{mytheo}(Energy-Preserving property)\label{epp}
			The two splitting schemes {formulated} in Algorithm \ref{algo} exactly preserve {the energy $H(x^0,v^0)$ of the CPD  \eqref{CPD equ}.}
		\end{mytheo}\begin{proof}
				{To this end, we first prove that the second subflow  of \eqref{subflows} exactly conserves the modified energy $\hat{H}$ with  the scheme
					\begin{equation}\label{me}
						\hat{H}(v,r):=\frac{1}{2}\abs{v}^2+\ln(r)=H(x,v).
					\end{equation}
					{For the second subflow, taking the inner product with $v$ of the second equality and using the other two equalities},  it is obtained that
					\begin{align*} {\frac{d}{dt}\left(\frac{1}{2}\abs{v}^2+\ln(r) \right)= v^\intercal  \dot{v}+\frac{\dot{r}}{r}=v^\intercal \frac{E(x)}{\exp(U(x))}r-\frac{\dot{x}^\intercal E(x)}{\exp(U(x))}r=0},
					\end{align*}
					which shows that $\frac{d}{d t}\hat{H}(v,r)=0$ and further yields \eqref{me}.
					
					Then we prove that} the two schemes preserve the modified energy \eqref{me} at the discrete level, i.e.,
				\begin{equation}\label{medl}
					\hat{H}(v^{n+1},r^{n+1})=\hat{H}(v^n,r^n),\ \ \text{for} \ \ n=0,1,2,\dots,T/h-1.
				\end{equation}
				It can be deduced from $v^{n+1}$ and $\ln(r^{n+1})$ in \eqref{S1-ESAV} that
				\begin{align*} \frac{1}{2}\abs{v^{n+1}}^2&=\frac{1}{2}\left[e^{h\tilde{B}(x^n)}v^n+h{F(\hat{x}^{n+\frac{1}{2}},\hat{r}^{n+\frac{1}{2}})}\right] ^\intercal \left[e^{h\tilde{B}(x^n)}v^n+h{F(\hat{x}^{n+\frac{1}{2}},\hat{r}^{n+\frac{1}{2}})}\right] \\
					&=\frac{1}{2}\abs{v^n}^2+\left[ he^{h\tilde{B}(x^n)}v^n+\frac{h^2}{2}{F(\hat{x}^{n+\frac{1}{2}},\hat{r}^{n+\frac{1}{2}})}\right] ^\intercal {F(\hat{x}^{n+\frac{1}{2}},\hat{r}^{n+\frac{1}{2}})}\\
					&=\frac{1}{2}\abs{v^n}^2-\ln (r^{n+1})+\ln (r^n),
				\end{align*}
				which indicates that {\eqref{medl} holds for S1-ESAV. It is clear that the modified energy-preserving property \eqref{medl} of S2-ESAV  can be proved} in the same way.
				
				{Finally,} based on the above proof and noting $\hat{H}(v,r)=H(x,v)$, {it is immediately concluded that our two methods satisfy} \begin{equation}\label{thm res}\hat{H}(v^n,r^n)=\hat{H}(v^0,r^0)=H(x^0,v^0)\ \ \text{for} \  \ n=1,2,\dots,T/h,\end{equation}  which completes the proof.
				\hfill $\blacksquare$
				
			\end{proof}
			
			{It should be pointed out that  the numerical results $r^n$ and $x^n$ produced by   Algorithm \ref{algo}
				usually do  not satisfy  $r^n=\exp(U(x^n))$. Thus  the statement \eqref{me} does not hold anymore for Algorithm \ref{algo}, i.e.,
				$\hat{H}(v^n,r^n)\neq H(x^n,v^n)$.  That's the reason why we prove the  result \eqref{thm res} in Theorem \ref{epp} instead of
				$H(x^n,v^n)= H(x^0,v^0)$.
			}
			
			\begin{mytheo}{(Global errors)
					Supposing that}  the nonlinear function $E(x): \mathbb{R}^3\rightarrow \mathbb{R}^3$ is sufficiently smooth,  there exists a sufficient small $0 < \alpha  \le 1$, such that when $0<h\le \alpha$, we have {
					{\begin{align*}& \textmd{S1-ESAV:}\qquad x^{n+1}-x(t_{n+1})=\mathcal{O}(h),\ \ x^{n+1}-x(t_{n+1})=\mathcal{O}(h), \\ &\textmd{S2-ESAV:}\qquad x^{n+1}-x(t_{n+1})=\mathcal{O}(h^2),\ \ x^{n+1}-x(t_{n+1})=\mathcal{O}(h^2),\end{align*}}
					where $n=0,1,2,\dots,T/h-1$ and the constants symbolized by
					$\mathcal{O}$ can depend on $T$ but not on $n, h$.}
			\end{mytheo}
			\begin{proof}
				{The proof is based on Taylor expansion and the local error analysis of splitting, which is omitted here for brevity.}\hfill $\blacksquare$
			\end{proof}
			
			{Besides the above results, there are some points which need to be noted}.
			\begin{rem}
				{From the scheme of the two methods, it follows that the methods keep the same property $r^{n+1}>0$  as the exact solution $r(t)>0$.}
				{It is also remarked that 	the proposed methods are explicit, and therefore they are more efficient than linearly implicit methods} in practical computations.
			\end{rem}
			\begin{rem}
				It is worth mentioning that the exponential function is rapidly increasing {and thus there may exist} a rapidly increase of errors which may  lead Algorithm \ref{algo} to lose efficiency. In order to avoid this {point}, we add a suitable positive constant $C=\abs{H(x^0,v^0)}$ in the exponential scalar auxiliary variable: $r(t)=\exp\left( \frac{U(x)}{C}\right)$ {and define $G(x,r):=\frac{E(x)}{\exp(U(x)/C)}r$, then we can }{obtain the modified S1-ESAV {(S1-MESAV) }in the form:}
				\begin{equation}
					\begin{aligned} &x^{n+1}=x^n+he^{h\tilde{B}(x^n)}v^n+\frac{h^2}{2}{G(\hat{x}^{n+\frac{1}{2}},\hat{r}^{n+\frac{1}{2}})},\ \ v^{n+1}=e^{h\tilde{B}(x^n)}v^n+h{G(\hat{x}^{n+\frac{1}{2}},\hat{r}^{n+\frac{1}{2}})},\\
						&\ln (r^{n+1})=\ln (r^n)-\frac{(x^{n+1}-x^n)^\intercal }{C}{G(\hat{x}^{n+\frac{1}{2}},\hat{r}^{n+\frac{1}{2}})}.
					\end{aligned}
				\end{equation}
				{In a same way, we can get the expression of modified S2-ESAV {(S2-MESAV)} and  the above two modified methods  are referred as SMESAVs. With the same arguments as above, it can be shown that} these two modified methods preserve the modified energy $\widehat{H}(v,r)=\frac{1}{2}\abs{v}^2+C\ln(r)$.
			\end{rem}
			
			\section{Numerical experiment}\label{ne}
			In {Section \ref{method}}, a novel class of explicit energy-preserving splitting schemes SESAVs and {SMESAVs were proposed for CPD}. To demonstrate their {numerical behaviour}  in accuracy, {cputime}, and energy conservation, we present a numerical experiment in this section. The linearly implicit splitting method S1-SAV \eqref{S1SAV} is used to make a comparison with SESAVs. The ``ode45'' function of MATLAB is used to get the reference solution.
			
			
			We consider the lower bound case of $U$ where the scalar potential is chosen as $U(x)=\frac{1}{100\sqrt{x_1^2+x_2^2}}$ to compare with S1-SAV and the magnetic field is chosen as  ({\cite{2020Long}}): {$B(x)={\frac{1}{\eps}}(0,0,\sqrt{x_1^2+x_2^2})^\intercal.$} For initial values we take $x(0)=(0,1,0.1)^\intercal$ and $v(0)=(0.09,0.05,0.2)^\intercal$.  Figure \ref{energy} shows the relative energy {errors}: $e\tilde{H}:=\frac{\abs{\tilde{H}(x^n,r^n)-H(x^0,v^0)}}{\abs{H(x^0,v^0)}}$ with {the scalar $r=\sqrt{U(x)+1}$ and }the modified energy {$\tilde{H}(v,r)=\frac{1}{2}\abs{v}^2+r^2-1$} for S1-SAV, $e\hat{H}:=\frac{\abs{\hat{H}(x^n,r^n)-H(x^0,v^0)}}{\abs{H(x^0,v^0)}}$ with {the scalar $r=\exp(U(x))$ and }the modified energy $\hat{H}(v,r)=\frac{1}{2}\abs{v}^2+\ln(r)$ for SESAVs, {and  $e\widehat{H}:=\frac{\abs{\widehat{H}(x^n,r^n)-H(x^0,v^0)}}{\abs{H(x^0,v^0)}}$ with {the scalar $r=\exp(U(x)/C)$, the constant $C=\abs{H(x^0,v^0)}$ and} the modified energy $\widehat{H}(v,r)=\frac{1}{2}\abs{v}^2+C\ln(r)$  for SMESAVs.} The global errors $error:=\frac{\abs{x^n-x(t_n)}}{\abs{x(t_n)}}+\frac{\abs{v^n-v(t_n)}}{\abs{v(t_n)}}$ are displayed in Figure \ref{ge}. Furthermore, Figure \ref{cpu} presents the cputime\footnote{{This test is conducted in a sequential program in MATLAB on a desktop (CPU: Intel (R) Core (TM) i7-8700 CPU @ 3.20 GHz, Memory: 8 GB, Os: Microsoft Windows 11 with 64bit).}}
			of S1-SAV and S1-ESAV.
			
			The following observations can be drawn from {Figures \ref{energy}-\ref{cpu}.  a)} {SESAVs and SMESAVs are energy-preserving, and the later ones hold a better energy behavior than the former especially in the case $\eps=1$.} b) S1-ESAV {behaves a same precision as S1-SAV but has} a better energy-preserving property. c) Figure \ref{cpu} illustrates that S1-ESAV outperforms S1-SAV in terms of computing efficiency. {In this numerical example, the performances of accuracy and cputime of SESAVs and SMESAVs are similar, and so we omit them for brevity.}
			
			\begin{figure}[tpb]
				\centering
				\begin{tabular}[c]{ccc}%
					\subfigure{\includegraphics[width=4.3cm,height=4.0cm]{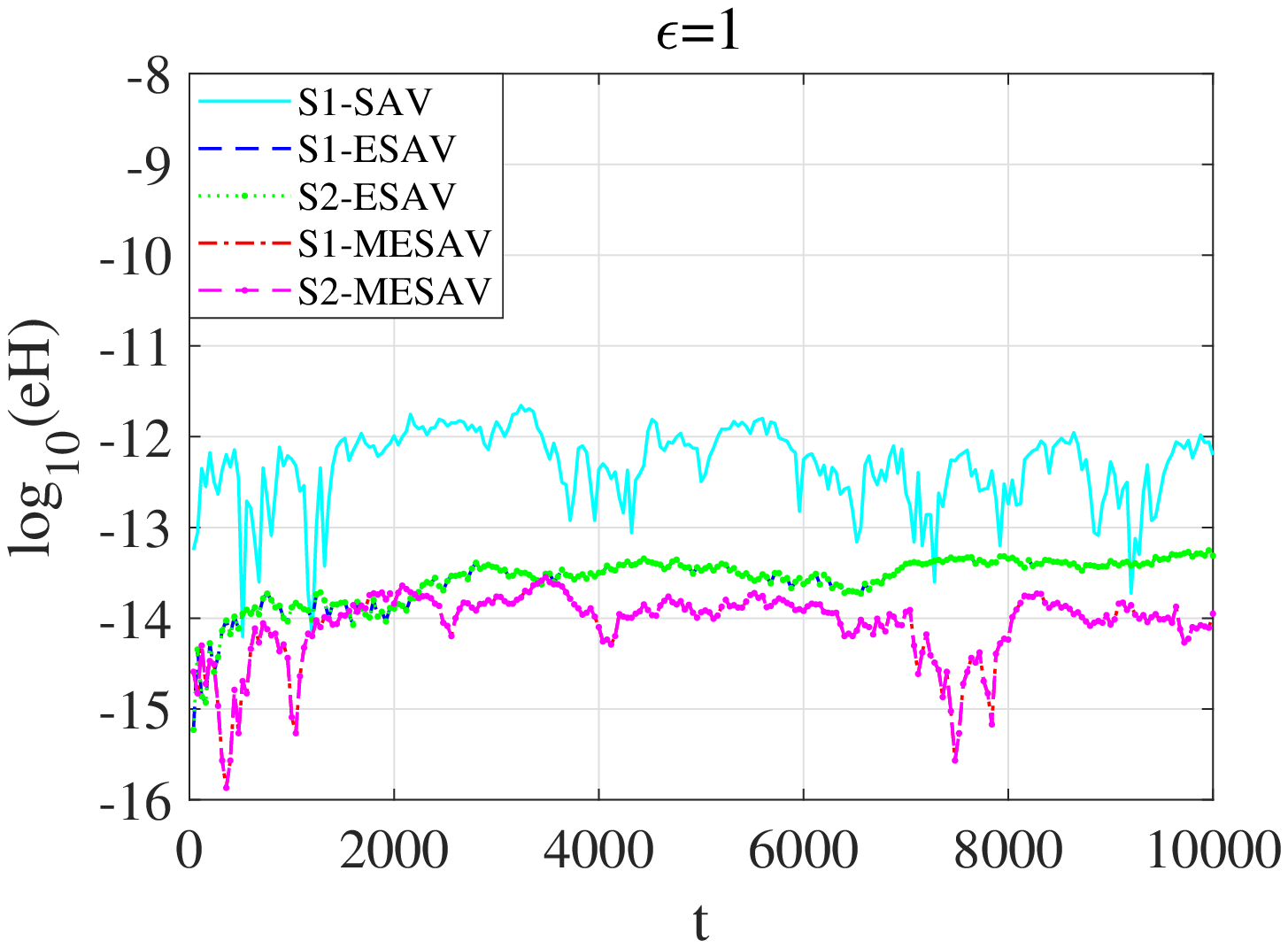}}			\subfigure{\includegraphics[width=4.3cm,height=4.0cm]{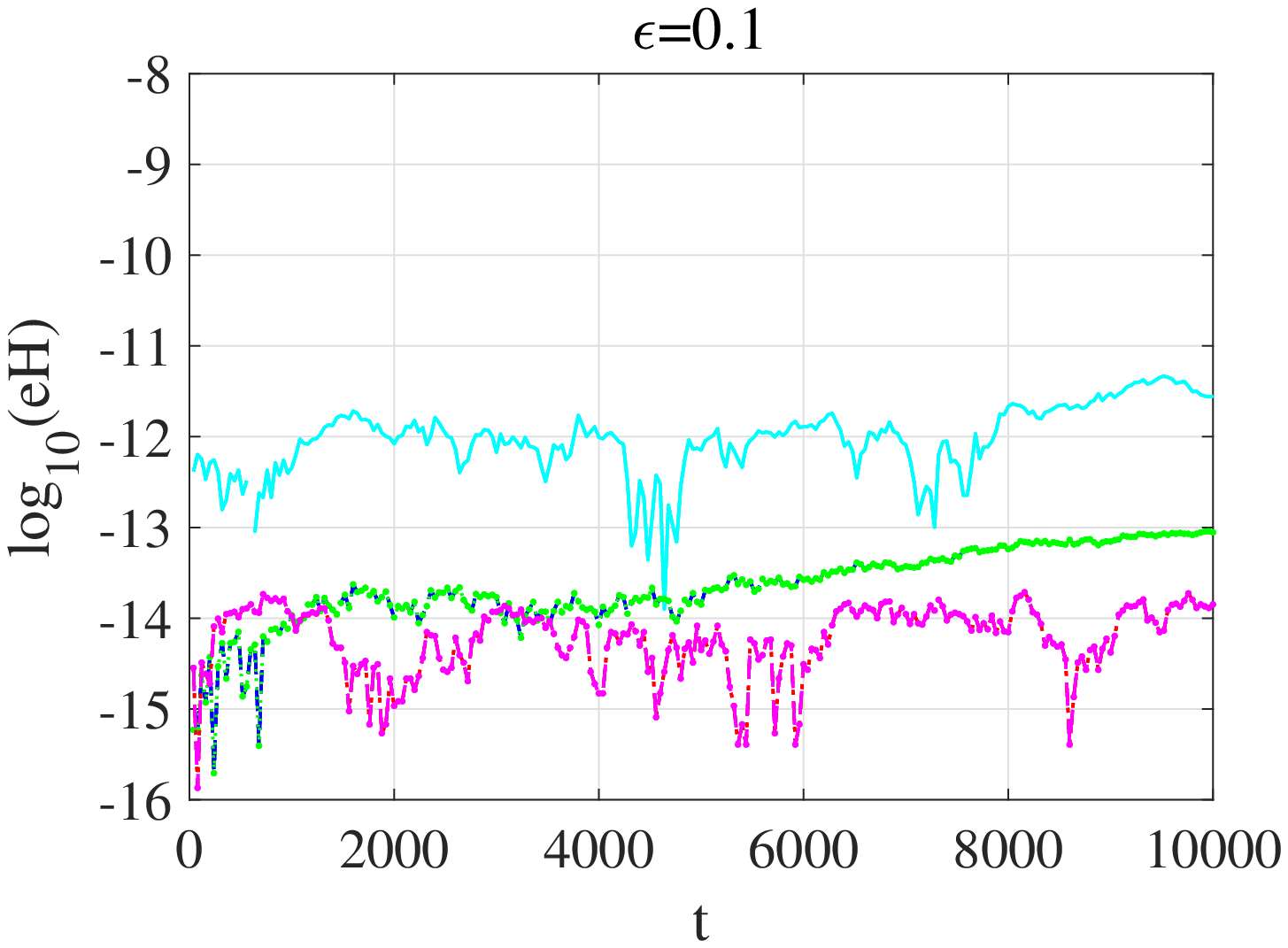}}
					\subfigure{\includegraphics[width=4.3cm,height=4.0cm]{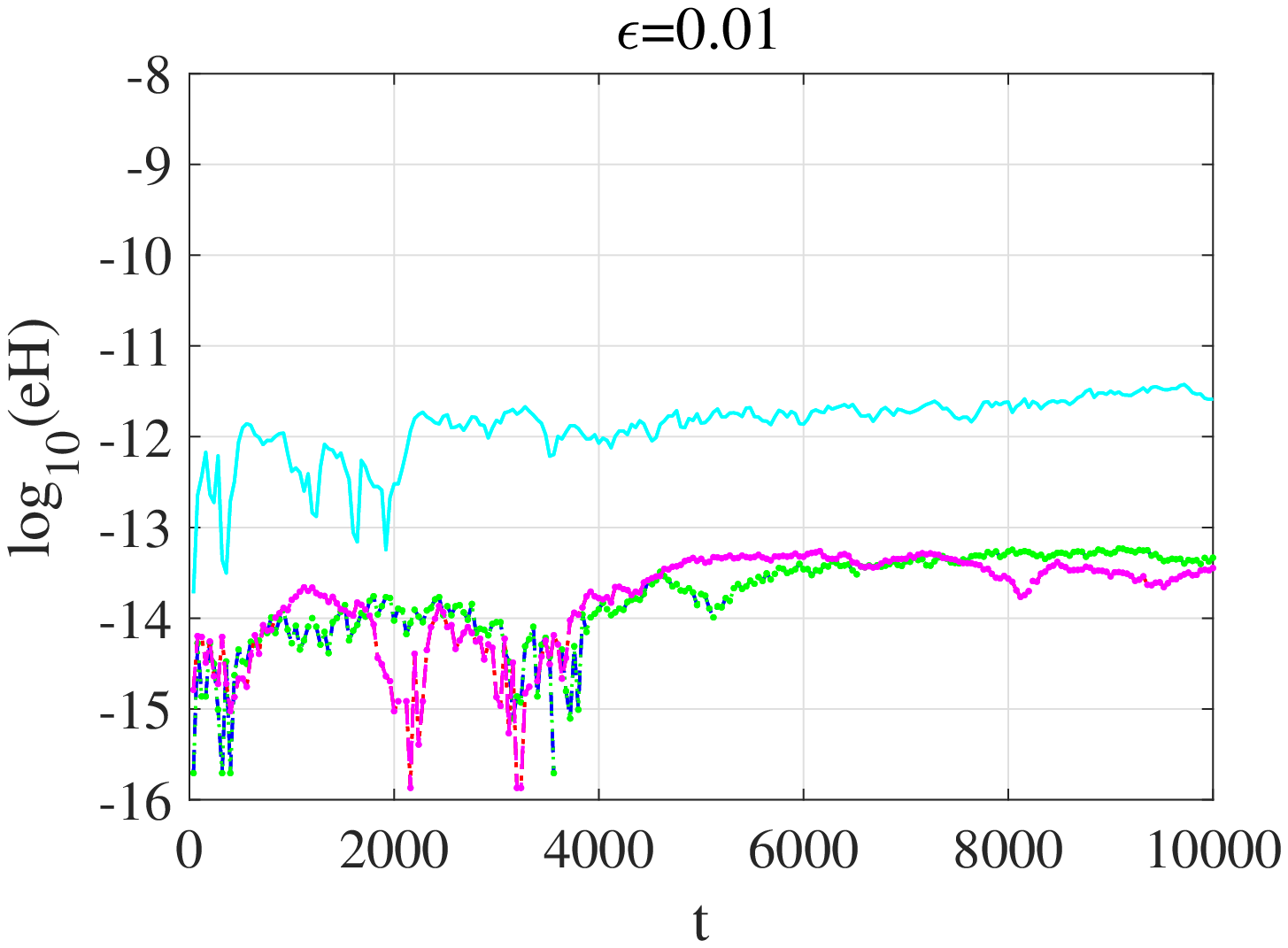}}		
				\end{tabular}
				\caption{Evolution of the energy error $e_H$ as function of time $t$. }\label{energy}
			\end{figure}
			
			\begin{figure}[tpb]
				\centering
				\begin{tabular}[c]{ccc}%
					\subfigure{\includegraphics[width=4.3cm,height=4.0cm]{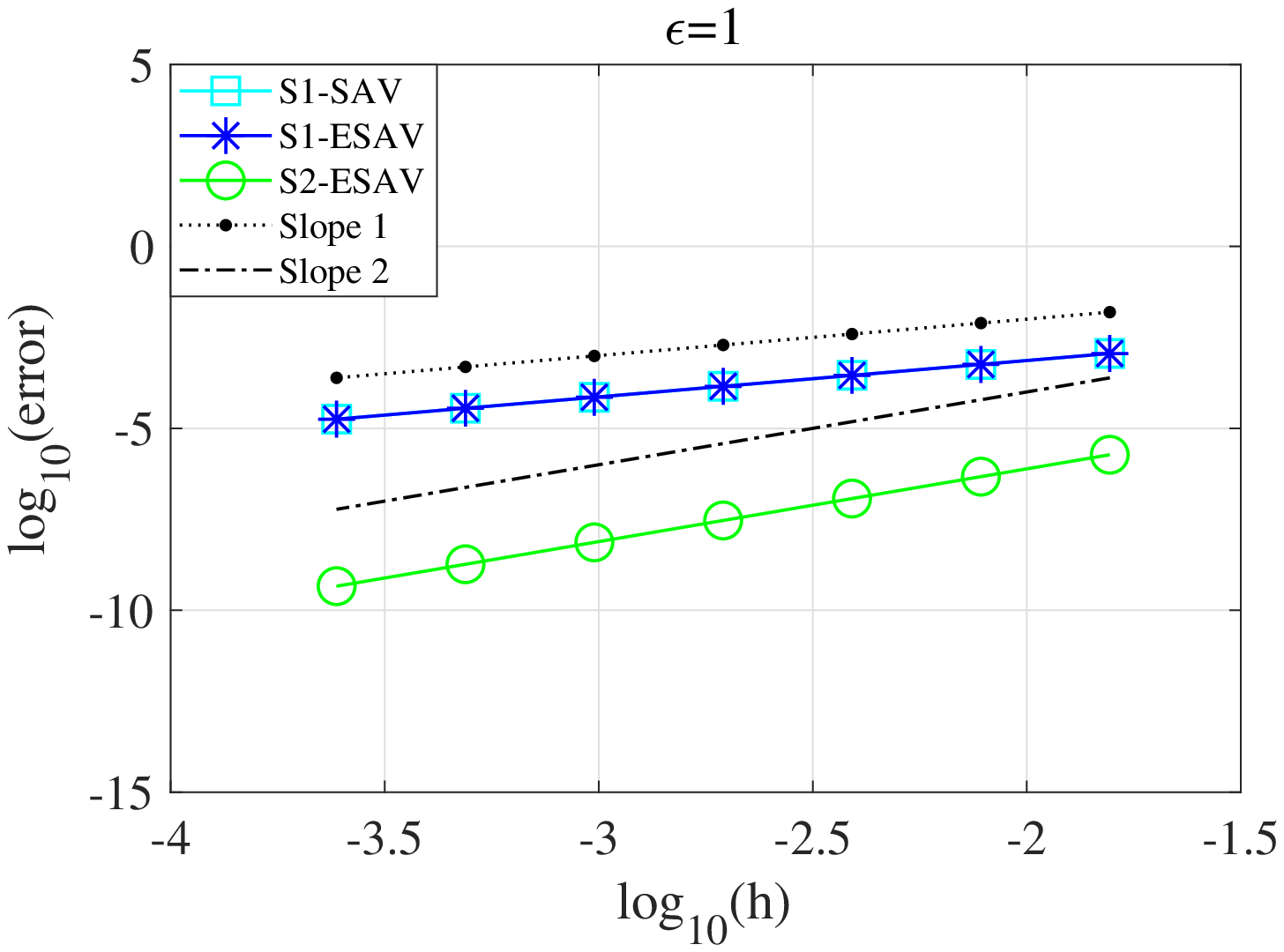}}		
					\subfigure{\includegraphics[width=4.3cm,height=4.0cm]{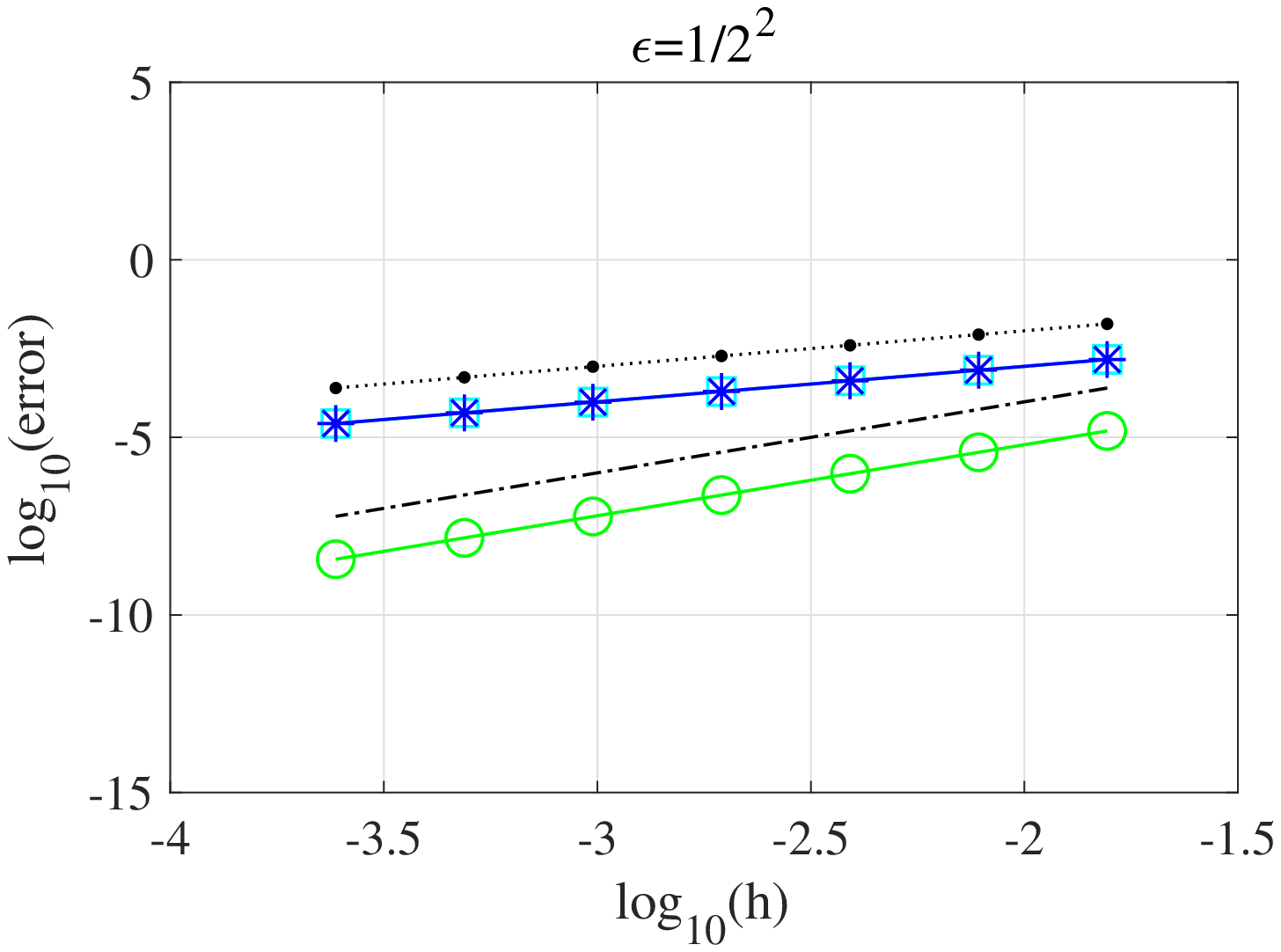}}
					\subfigure{\includegraphics[width=4.3cm,height=4.0cm]{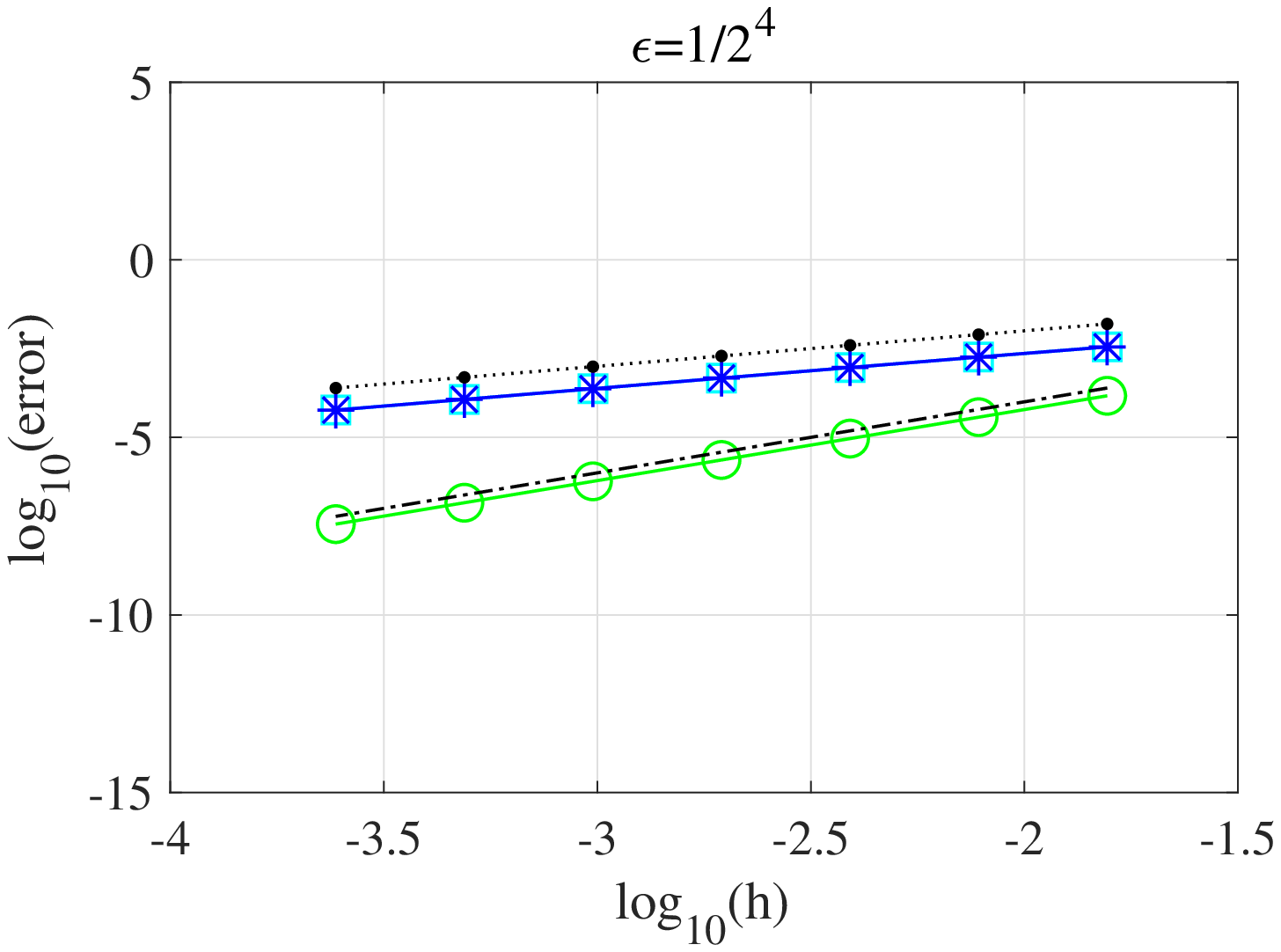}}			
				\end{tabular}
				\caption{The global errors  $error$ with {$t=1$} and $h=1/2^k$ for $k=6,\ldots,12$ under different $\eps$. }\label{ge}
			\end{figure}
			
			\begin{figure}[tpb]
				\centering
				\begin{tabular}[c]{ccc}%
					\subfigure{\includegraphics[width=4.3cm,height=4.0cm]{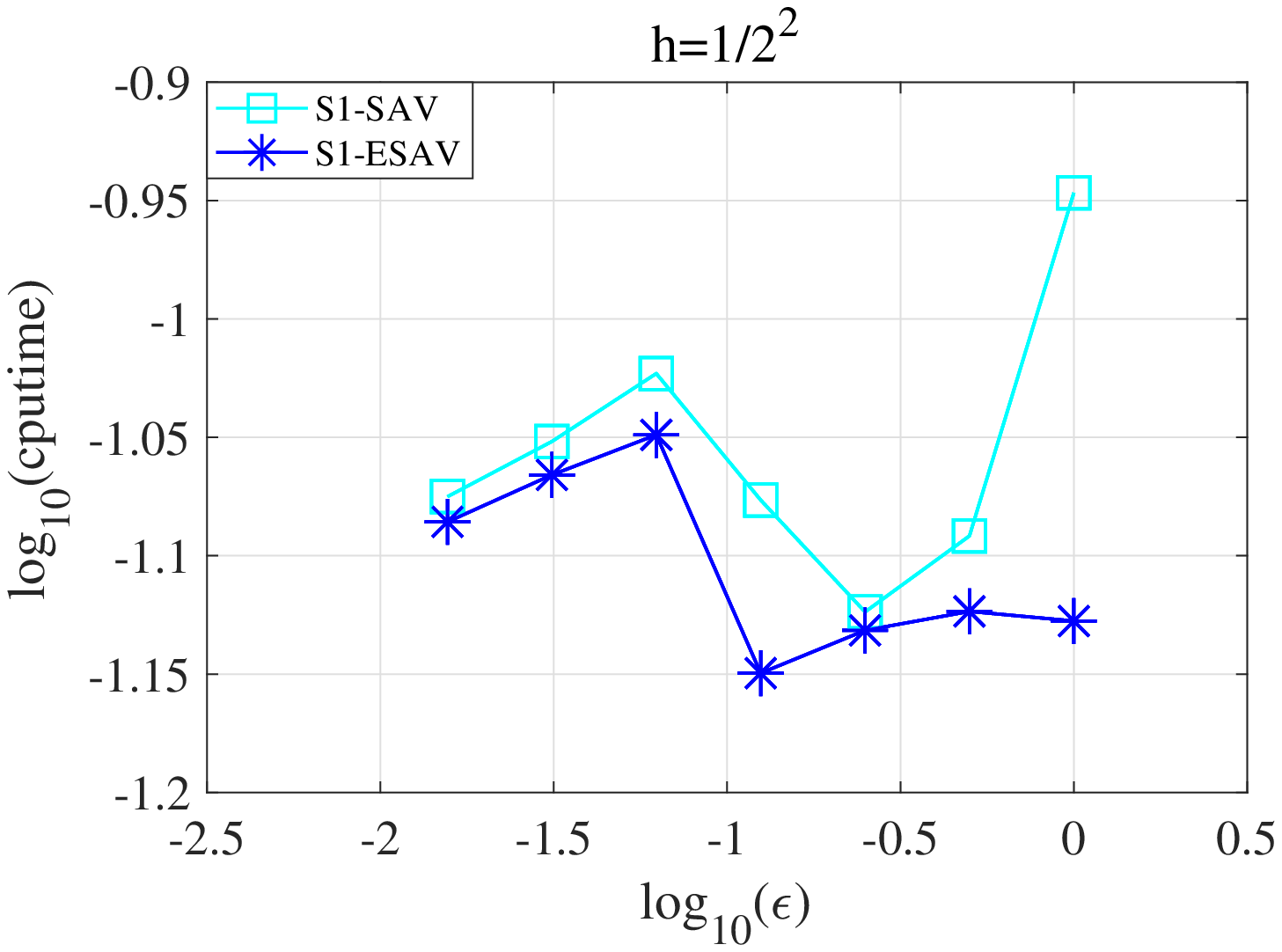}}		
					\subfigure{\includegraphics[width=4.3cm,height=4.0cm]{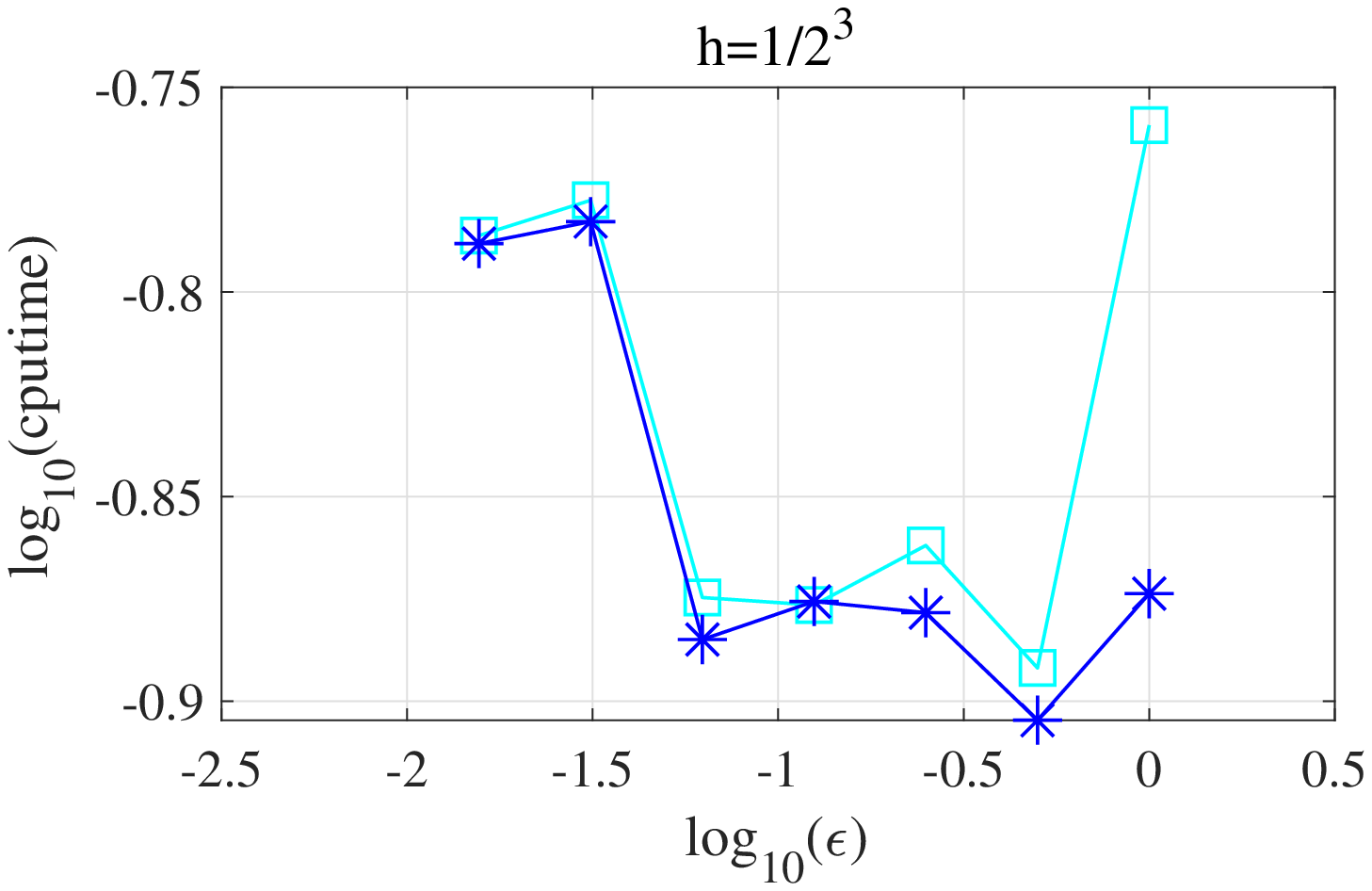}}
					\subfigure{\includegraphics[width=4.3cm,height=4.0cm]{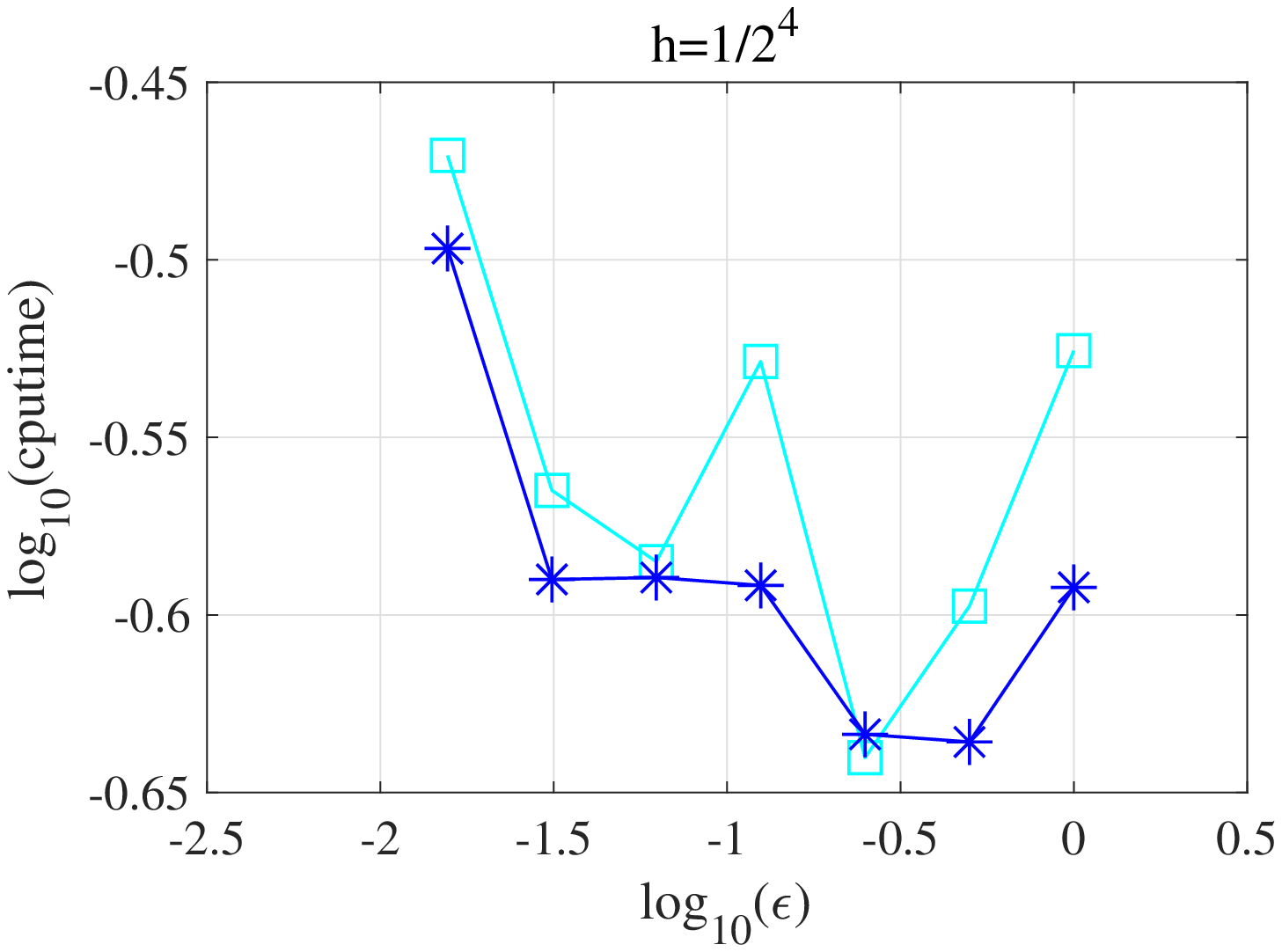}}			
				\end{tabular}
				\caption{The cputime with {$T=100$} and $\eps=1/2^k$ for $k=0,1,\ldots,6$ under different $h$. }\label{cpu}
			\end{figure}
			\section{Conclusion}\label{conc}
			In this letter, we have proposed two  energy-preserving splitting {methods} (SESAVs) for the charged-particle dynamics.
			{It was shown that these SESAVs are explicit and exactly preserve the energy of the charged-particle dynamics. The} accuracy of SESAVs was also presented. A numerical experiment was carried out  to illustrate the accuracy and energy conservation of these methods.

		\end{document}